\documentclass[10pt,a4paper]{article}
\usepackage[utf8]{inputenc}
\usepackage{amsmath,amsfonts,amssymb,mathtools,mathrsfs,%
enumitem}
\usepackage{amsthm}
\usepackage{comment}
\usepackage{hyperref}
\usepackage[backend=biber,
style=ieee
]{biblatex} 
\usepackage[margin=4cm]{geometry}
\usepackage{xcolor}
\usepackage{authblk}

\newtheorem{theorem}{Theorem}
\newtheorem{proposition}[theorem]{Proposition}

\newtheorem{lemma}[theorem]{Lemma}

\theoremstyle{remark}
\newtheorem{remark}[theorem]{Remark} 
\theoremstyle{definition}
\newtheorem{example}[theorem]{Example}

\newcommand{\nulllagrangianlemma}{{\cite[Lem.~27]{OurCounterexamplePaper}}}

\definecolor{dkgreen}{rgb}{0,0.4,0}
\definecolor{dkred}{rgb}{0.8,0.0,0}

\title{%
Sufficient conditions for the absence of relaxation gaps in state-constrained optimal control}

\date{}

\newcommand{\R}{\mathbb R}

\DeclareMathOperator{\dist}{\operatorname{dist}}

\newcommand{\domain}{\Omega}
\newcommand{\T}{T}
\newcommand{\controls}{U}
\newcommand{\g}{g}
\newcommand{\f}{f}
\newcommand{\F}{F}

\newcommand{\initial}{{\mathbf{x}_0}}
\newcommand{\normalL}{L}

\newcommand{\mincurves}{M_{\mathrm{c}}}
\newcommand{\minarcs}{M_{\mathrm{y}}}
\newcommand{\minoccupation}{M_{\mathrm{o}}}

\newcommand{\destination}{X}

\newcounter{Ccnt}
\makeatletter
\newcommand\C[1]{%
\@ifundefined{C-#1}%
  {\stepcounter{Ccnt}\expandafter\xdef\csname C-#1\endcsname{\arabic{Ccnt}}}%
  {}%
c_{\csname C-#1\endcsname}}
\makeatother 

\begin{document}

\author[1]{Nicolas Augier}
\author[1,2]{Milan Korda}
\author[2,3]{Rodolfo Rios-Zertuche}
\affil[1]{\footnotesize 
LAAS-CNRS, Toulouse, France
}
\affil[2]{\footnotesize 
Faculty of Electrical Engineering, Czech Technical University in Prague, Czechia
}
\affil[3]{\footnotesize Department of Mathematics and Statistics, UiT The Arctic University of Norway}

\maketitle

\begin{abstract}
 This work presents new 
 sufficient conditions for the absence of a gap corresponding to %
Young measure and occupation measure relaxations for constrained optimal control problems. Unlike existing conditions, these sufficient conditions do not rely on convexity of the Lagrangian or the set of admissible velocities. We use these conditions to derive new bounds for the size of the relaxation gap.
\end{abstract}

\section{Introduction}
\label{sec:intro}

\if{
New outline:
\begin{itemize}
    \item practical importance of occupation measure relaxation
    \item statement of the problem (when is $\mincurves=\minoccupation$?)
    \item formulation of classical problem
    \item formulation of occupation measure relaxation
    \item formulation of young measure relaxation, explaining it is intermediary and of theoretical importance
    \item Lemma 1
    \item organization of the paper (contributions and outline)
\end{itemize}
}\fi

The occupation measure relaxation has the power of turning highly-nonlinear optimal control problems into  computationally-tractable hierarchy of semi-definite programs \cite{OCP08,korda2018moments,augier2024symmetry}. 
This relaxation was hinted at already in the seminal work of L.C.~Young~\cite{young_lectures}, put on a sound theoretical footing by Lewis and Vinter~\cite{lewis1980relaxation, V93}, and finally exploited computationally in conjunction with moment-sum-of-squares hierarchies applicable when the problem data are semialgebraic, %
starting with the work~\cite{OCP08}, where it was developed for optimal control of ODEs, and was eventually generalized to higher-dimensional problems involving PDEs in \cite{korda2018moments}; see also \cite{fantuzzi2022sharpness,henrion2024occupation}. Since then, this approach has proved to be remarkably successful, allowing for powerful tools of convex optimization to be deployed for tackling this challenging non-convex  problem and many of its variations.  

The success of the occupation measure relaxation hinges on the equality of the optimal value of the original problem and of the relaxed one. Unfortunately, this equality does not hold without further assumptions; see \cite{OurCounterexamplePaper,palladino2014minimizers} for counterexamples.

At present, the only sufficient conditions known to the authors are due to~\cite{V93} and revolve around the convexity of the Lagrangian and of the set of admissible velocities. In this note we show that these conditions are far from necessary and propose a complementary set of sufficient conditions that do not rely on convexity and have a strong practical appeal. In addition to not relying on convexity, the proposed conditions do not require that the state or control constraint sets be compact, contrary to the setting of~\cite{V93}. These conditions are based on a generalization of the Filippov--Wa\v zewsky theorem \cite{frankowska2000filippov} and a recent superposition result due to Bernard~\cite{bernard2008generic}.

In the higher-dimensional case of optimal control of partial differential equations, the question of the existence of a relaxation gap is also a very active research topic.
Very few criteria exist in full generality. In particular, it turns out that for general higher-dimensional variational problems, convexity conditions analogous to those of Theorem \ref{VINTER_THM} are not sufficient in order to ensure the absence of relaxation gap (see e.g.~\cite{korda2022gap}), and they instead need to be strengthened  considerably; see e.g.~\cite{henrion2024occupation}. Other examples of gaps in higher-dimensional settings can be found e.g.~in~\cite{fantuzzi2022sharpness}.

\section{Problem statement}
Let $\domain$ be a subset of $\R^n$, and choose an initial point $\initial\in\domain$, a target set $\destination\subset\domain$, a fixed amount of time $\T>0$, and a set of controls $\controls\subset\R^n$. Let also $\normalL\colon\R\times\domain\times\controls\to\R$ be a function that will serve as Lagrangian density, $\f\colon\R\times\domain\times\controls\to\R^n$ is the controlled vector field, and $\g\colon\R^n\to\R$ a function giving the terminal cost.

The original (classical) problem that we want to consider and eventually relax is $\mincurves=\mincurves(\domain,\destination)$ defined by
\begin{alignat}{2} \label{min:orig}
 \mincurves= &\!\inf\limits_{u}  & \;&\displaystyle\int_0^\T \normalL(t,\gamma(t),u(t))dt +g(\gamma(\T)) \\ 
 &\textrm{s.t.}& & \text{$u\colon[0,\T]\to\controls$ measurable},\nonumber\\
 &&& \gamma\colon[0,T]\to\Omega\;\text{absolutley continuous}, \nonumber\\
 &&& \gamma'(t)=f(t,u(t),\gamma(t))\;\text{a.e.}\;t\in[0,\T], \nonumber\\
 &&&\gamma(0)=\initial,\;\gamma(\T)\in\destination,\nonumber
\end{alignat}
where we are minimizing the sum of the integral of $\normalL$ and the terminal cost  given by $\g$ over all measurable controls $u\colon[0,\T]\to\controls$ that determine admissible curves $\gamma\colon[0,\T]\to\domain$ by
\[\gamma(t)=\initial+\int_0^tf(s,\gamma(s),u(s))\,ds,\]
that end in $\gamma(\T)\in\destination$. %

Throughout, we will assume that \emph{the image $f(K)$ of every compact set $K\subseteq [0,\T]\times\domain\times\controls$ is compact} and that \emph{$\mincurves$ is finite}; in particular, this implies that there is at least one admissible curve joining $\initial $ and $\destination$. 
We will further specify  $\domain$, $\destination$, $\controls$, $\normalL$, $\f$, and $\g$ later on. 
We consider  the \emph{occupation measure relaxation} $\minoccupation=\minoccupation(\domain,\destination)$ given by
\begin{equation}\label{min:occupation}
    \minoccupation=\inf_{\mu,\mu_\partial} \int_{[0,\T]\times\domain\times\controls} \hspace{-7mm}\normalL(t,x,u)\,d\mu(t,x,u)
    +\int_{\domain}\g(x)\,d\mu_\partial(x),
\end{equation}
where the infimum is taken over all positive Borel measures $\mu$ on $[0,\T]\times\domain\times\controls$ and all Borel probability measures $\mu_\partial$ on $\destination$ satisfying that 
\begin{equation*}%
\int_{[0,T]\times\domain\times\controls}\|f(t,x,u)\|d\mu+\int_{\destination}\|x\|\,d\mu_\partial(x)
\end{equation*}
be finite,
and, for all $C^\infty$, compactly-supported functions $\phi\colon[0,\T]\times\domain\to\R$,
\begin{equation*}
    \int_{[0,\T]\times\domain\times\controls}\frac{\partial\phi}{\partial t}+\frac{\partial\phi}{\partial x}\cdot f(t,x,u)\,d\mu
    +\phi(0,\initial)-\int_\controls \phi(\T,x)d\mu_\partial(x)=0.
\end{equation*}

As an auxiliary tool, albeit itself of theoretical interest, we will use the well-known \emph{Young measure relaxation} $\minarcs=\minarcs(\domain,\destination)$ defined as
\begin{equation}
    \label{min:relaxed}
    \minarcs=\inf_{(\nu_t)_{t}}\int_0^\T\int_\controls
    \normalL(t,\gamma(t),u)d\nu_t(u)dt
    +g(\gamma(\T)),
\end{equation}
where the infimum is taken over all families $(\nu_t)_{t\in[0,T]}$ of compactly-supported Borel probability measures $\nu_t$ on the set of controls $\controls$ such that the integral curve
\[ \gamma(t) =\initial +\int_0^t\int_\controls f(s,\gamma(s),u)\,d\nu_s(u)\,ds\]
remains in $\domain$ and
satisfies $\gamma(\T)\in\destination$.
Here we are minimizing over all \emph{Young measures} $(\nu_t)_{t\in[0,T]}$. Taking one such measure, we define $\gamma$ to be its integral curve starting at $\initial$ (which means that $\gamma(0)=\initial$ and, for almost every $t\in[0,T]$, its derivative $\gamma'(t)=\int_Uf(t,\gamma(t),u)\,d\nu_t(u)$ is the average speed  provided by $\f$ over $\nu_t$).

\subsection{Previously known results}
\label{sec:previous}
In contrast with the desirable equalities $\minoccupation=\minarcs=\mincurves$, in general we only have the following relations, which are proved in Appendix \ref{appendix}.
\begin{lemma}\label{lem:easyineqs}
We have 
\begin{equation*}%
    \minoccupation=\minarcs\leq\mincurves.
 \end{equation*}

\end{lemma}

To date, the only known %
properties 
ensuring the absence of a relaxation gap for Problems \eqref{min:occupation} and \eqref{min:relaxed}
are the following convexity assumptions 
proposed in \cite{V93}.
Let, for $t\in\R$, $x\in\overline\Omega$, $v\in \R^n$,
\begin{equation*}
    \bar\normalL(t,x,v)=\inf\{ \normalL(t,x,u): 
    f(t,x,u)=v,\;
    u\in \controls\}.
\end{equation*}
Theorem \ref{VINTER_THM} requires these assumptions:
\begin{enumerate}[label=V\arabic*.,ref=V\arabic*,leftmargin=*]
    \item \label{vinter:first}The functions $\bar \normalL$ and $g$ are lower semicontinuous.
    \item $\overline\Omega$ and $\overline X$ are compact sets.
    \item For $(t,x)\in [0,\T]\times\overline\domain$, the sets $f(t,x,U)$ are compact, convex sets, and the graph of the set-valued map $(t,x)\mapsto f(t,x,U)$ over $[0,\T]\times \overline\Omega$ is compact. 
    \item For each $(t,x)\in [0,T]\times \R^n$, the function $v\mapsto\bar \normalL(t,x,v)$ restricted to $v\in f(t,x,U)$ is a convex function.
    \item There exists an admissible trajectory $\bar x\colon [0,\T]\to\overline\Omega$ such that
    \[\int_{0}^{\T}\bar \normalL(t,\bar x(t),\bar x'(t))\,dt +g(\bar x(\T))\]
    is finite.
    \label{vinter:last}
\end{enumerate}

\begin{theorem}[{\cite{V93}}]\label{VINTER_THM}
    Under hypotheses \ref{vinter:first}--\ref{vinter:last}, the infimum $\mincurves(\overline\domain,\overline\destination)$  is achieved by a measurable control $u$, and we have $\mincurves(\overline \Omega,\overline X)=
    \minarcs(\overline \Omega,\overline X)=\minoccupation(\overline \Omega,\overline X)$. %
\end{theorem}

However, these conditions are far from necessary, as can be easily seen in the following example.

\begin{example} Take $T=1$, $\domain=\destination=\R$, $f(t,x,u)= u$, $U = [-1,1]$, $L(t,x,u) = (u^2-1)^2 + x^2$, $\initial = 0$. In this case we have $\minoccupation=\minarcs=\mincurves=0$, while the Lagrangian $L$ is not convex w.r.t. $u$. 
Indeed, a minimizing occupation measure is $\mu=\tfrac12\delta_0(x)\otimes (\delta_{+1}(u)+\delta_{-1}(u))$, which gives $\minoccupation=0$. On the other hand, a minimizing sequence of measurable controls can be given by $u_N(t)=(-1)^{\lfloor Nt\rfloor}$, where $\lfloor \cdot\rfloor$ is the 
floor function and $N\in\mathbb N$; this gives $\mincurves=0$.
\end{example}

\subsection{Contributions and outline}
In this paper, we propose sufficient conditions for $\minoccupation=\minarcs=\mincurves$, which amount to the following cases:
\begin{enumerate}
    \item\label{it:firstcontrib} The infima taken over the interior and over the closure of the  domain coincide for the Young measure relaxation; see Theorem \ref{cor:int} in Section \ref{sec:interiorboundary}.
    \item\label{it:secondcontrib} The control structure satisfies an inward pointing condition at the boundary of the domain; see Theorem \ref{cor:IPC} in Section \ref{sec:inwardpointing}.
\end{enumerate}

Condition \ref{it:firstcontrib} is particularly appealing from a practical perspective as it amounts to the infimum of the problem being stable under an infinitesimal perturbation of the domain's boundary.

The application of the results of Section~\ref{sec:interiorboundary} allows us to obtain in Section~\ref{sec:gapbound}  explicit estimates of the gap in the case where the sufficient conditions are not satisfied. They are obtained by replacing the sets $\Omega$ and $X$ by suitable inner approximations. This is interesting in practice since the corresponding upper bound can be computed via the moment-sums-of-squares algorithm from~\cite{OCP08}.

\section{Interactions between the interior and the boundary of the domain}
\label{sec:interiorboundary}

In this section we will be concerned with Filippov--Wa\v zewsky-type results in a state-constrained setting. These have a long history; see for example \cite{frankowska2000filippov,
CKR22,
BETTIOL_BRESSAN,BETTIOL12}. 
In this context, the absence of a relaxation gap necessitates certain refined properties of the controlled vector field near the boundary of the constraint set $\Omega$.

    For $x\in\R^n$ and $r>0$, denote by $\overline B(x,r)$ the closed ball of radius $r$ centered at $\xi_0$.
    Denote also by $\dist_H$ the Hausdorff distance between two sets.
    Let $L^1([0,\T];\R_+)$ be the set of integrable, nonnegative functions
    defined on $[0,\T]$.

Consider the following assumptions on a set-valued function $F\colon\R\times\R^n\rightrightarrows\R^n$:
\begin{enumerate}[label=FW\arabic*.,ref=FW\arabic*,leftmargin=*]
    \item \label{FW:first} (Boundedness) There is $\lambda\in L^1([0,\T];\R_+)$ with 
    \[\sup_{v\in F(t,x)}\|v\|\leq \lambda(t)(1+\|x\|),\qquad\text{a.e.~$[0,\T]$.}\]
    \item (Local Lipschitzity) Suppose that for all $R>0$ there is $k_R\in L^1([0,T];\R_+)$ such that, for almost all $t\in[0,\T]$ and all $x,y\in B(0,R)$,
    \[\dist_H(F(t,x),F(t,y))\leq k_R(t)\|x-y\|.\]
    \label{FW:last}
\end{enumerate}
\begin{theorem}[No gap on open domains, and on some closed domains]
    \label{cor:int}
    Assume that $\domain$ and $X\subseteq \domain$ are open.
    Pick also a control structure $\f$, an initial point $\initial\in\domain$, a measurable Lagrangian density $\normalL(t,x,u)$ continuous in $x$, and a continuous terminal cost function $\g$.
    Define $\F$ by
    \begin{equation}\label{eq:defUsefulF}
        \F(t,x)=\{(f(t,x,u),L(t,x,u)):u\in \controls\},
    \end{equation}
    for $t\in[0,\T]$ and $x\in\domain$,
    and assume that it verifies \ref{FW:first}--\ref{FW:last}.
    Then we have:
    \[ \mincurves(\domain,\destination)= \minarcs(\domain,\destination)= \minoccupation(\domain,\destination).\]
    
    Additionally, if  $\minarcs(\domain,\destination)=\minarcs(\overline\domain,\destination)$ holds then $\mincurves(\overline\domain,\destination)=\minarcs(\overline\domain,\destination)=\minoccupation(\overline\domain,\destination)$. 
    
    Finally, if $\minarcs(\domain,\destination)=\minarcs(\overline\domain,\overline\destination)$ holds, then 
     $\mincurves(\overline\domain, \overline\destination)= \minarcs(\overline\domain,\overline\destination) =\minoccupation(\overline\domain,\overline\destination)$.
\end{theorem}
To prove Theorem \ref{cor:int}, we recall the following well-known result.
\begin{theorem}
    [Filippov-Wa\v zewski~{\cite[Th.~2.3]{frankowska2000filippov}}]
    \label{thm:FW}
    Assume that \ref{FW:first}--\ref{FW:last} are verified.
    Then, for every $\varepsilon>0$ and every absolutely continuous curve $x\colon[0,T]\to\R^n$ satisfying $x'(t)\in \operatorname{conv}F(t,x(t))$ for almost every $t\in[0,T]$ there is a curve $y\colon[0,T]\to\R^n$ satisfying $y'(t)\in F(t,y(t))$ for almost every $t\in[0,T]$, $x(0)=y(0)$, and
    \[\|x(t)-y(t)\|\leq \varepsilon,\qquad t\in [0,T].\]
\end{theorem}
\begin{proof}[Proof of Theorem \ref{cor:int}]
    We know from Lemma \ref{lem:easyineqs} that $\mincurves(\domain,\destination)\geq\minarcs(\domain,\destination)=\minoccupation(\domain,\destination)$. Let us show that $\mincurves(\domain,\destination)\leq\minarcs(\domain,\destination)$. Let $\varepsilon>0$.
    Take a contender $(\nu_t)_{t\in[0,\T]}$ in \eqref{min:relaxed}. Let $\gamma$ be the absolutely continuous curve satisfying $\gamma(0)=\initial$ and $\gamma'(t)=\int_\controls f(t,\gamma(t),u)\,d\nu_t(u)$ for almost every $t\in[0,\T]$. 
    By the classical lift extension procedure detailed in \nulllagrangianlemma{}, the system in \eqref{min:relaxed} is equivalent to a system with null Lagrangian in one dimension higher.
    Using the notations from \nulllagrangianlemma{} and its proof, translate the problem as in the lemma from a situation having data  $P=(\domain,\initial,\destination,\T,\controls,\normalL,\g,\f)$ to a new situation with data $\tilde P=(\tilde\domain,\tilde{\mathrm{x}}_0,\tilde \destination,\T,\controls,0,\tilde\g,\tilde\f)$ with $\tilde\domain=\domain\times\R$, $\tilde \destination=\destination\times\R$, $\tilde \normalL=0$, $\tilde g(x)=x_{n+1}+\g(x)$, $\tilde{\mathbf x}_0=(\initial,0)$, and
    $\tilde\f(t,x,u)=(\f(t,x,u),\normalL(t,x,u))$.
    Let $\tilde\gamma$ be the curve corresponding to $\gamma$, which is given by 
    \[\tilde\gamma(t)=(\gamma(t),\int_0^t\normalL(s,\gamma(s),u(s))\,ds).\]
    It follows that $\tilde\gamma$ satisfies 
    \[\tilde\gamma'(t)=\left(\int_\controls f(t,\gamma(t),u)d\nu_t(u),\int_\controls L(t,\gamma(t),u)d\nu_t(u)\right)\in \operatorname{conv}\F(t,\tilde\gamma(t))\] 
    for almost every $t\in[0,\T]$. %
    Apply Theorem \ref{thm:FW} with $x=\tilde\gamma$ to obtain a curve $y$ satisfying $\|\tilde\gamma(t)-y(t)\|\leq \varepsilon$ and $y'(t)\in  F(t,y(t))$. Project $y$ into the first $n$ coordinates to get a curve $\eta(t)=\pi(y(t))$ that satisfies $\eta'(t)\in \pi( \F(t,y(t)))=f(t,\eta(t),\controls)$ for almost every $t\in[0,\T]$ verifying
    \begin{equation*}
        \|\gamma(t)-\eta(t)\|=\|\pi(\tilde\gamma(t))-\pi(y(t))\| 
        \leq \|\tilde\gamma(t)-y(t)\|
        \leq \varepsilon.
    \end{equation*}
    Observe that, since $X$ is open and $\gamma(T)\in X$, this inequality means that, if we take $\varepsilon$ small enough that the ball $B(\varepsilon,\gamma(T))$ be contained in $X$, then also $\eta(T)\in X$.
    Applying the Measurable Selection Theorem \cite{aubinfrankowska} to get a measurable function $u(t)$ with $\eta'(t)=f(t,\eta(t),u(t))$ for almost every $t\in[0,\T]$,
    \begin{align*}
        &\left|\int_0^\T \int_\controls \normalL(t,\gamma(t),u)\,d\nu_t(u)\,dt-\int_0^\T\normalL(t,\eta(t),u(t))\,dt\right| \\
         &=\left|\tilde\gamma_{n+1}(\T)-\tilde\gamma_{n+1}(0)-y_{n+1}(\T)+y_{n+1}(0)\right|\\
         &=\left|\tilde\gamma_{n+1}(\T)-y_{n+1}(\T)\right|\leq \|\tilde\gamma(\T)-y(\T)\|\leq \varepsilon.
    \end{align*}
    Here, we have used the fact that $\tilde\gamma_{n+1}(0)=0=y_{n+1}(0)$, according to the equivalence explained in the proof of \nulllagrangianlemma. This shows that $|\mincurves(\domain,\destination)-\minarcs(\domain,\destination)|\leq \varepsilon$. Since this is true for arbitrary $\varepsilon>0$, this proves the first claim of the corollary.

    For the other statements of the corollary, we have, when $\minarcs(\domain,\destination)=\minarcs(\overline\domain,\destination)$,
\begin{equation*}
    \minarcs(\domain,\destination) =\minarcs(\overline\domain,\destination)\leq \mincurves(\overline\Omega,\destination)
    \leq \mincurves(\domain,\destination) = \minarcs(\domain,\destination),
\end{equation*}
which, together with Lemma \ref{lem:easyineqs}, gives the first set of equalities. Similarly, when 
$\minarcs(\domain,\destination)=\minarcs(\overline\domain,\overline\destination)$, we get the other set of equalities from
\begin{equation*}
    \minarcs(\domain,\destination) =\minarcs(\overline\Omega,\overline\destination)\leq \mincurves(\overline\domain,\overline\destination) 
    \leq \mincurves(\domain,\destination) = \minarcs(\domain,\destination). \qedhere
\end{equation*}
\end{proof}

\section{Inward pointing conditions}
\label{sec:inwardpointing}

Starting with \cite{frankowska2000filippov}, a few papers (see e.g.~\cite{CKR22,bib6, BETTIOL_BRESSAN,BETTIOL12,franceschi2020essential}) exploited a (Soner type) \emph{inward pointing condition} on the available vector fields on the boundary to obtain a Filippov--Wa\v zewsky type of result in the presence of domain constraints. The inward pointing condition  is also known as ``the absence of characteristic points'' in the context of sub-Riemannian geometry~\cite{franceschi2020essential}.

    A set-valued map $\mathscr I\colon\partial\Omega\to\R^n$ is called a \emph{uniformly hypertangent conical field} if there exist $\varepsilon>0$ and $\delta>0$, such that for every $x_0\in \partial\Omega$, $\mathscr I(x_0)$ is a closed convex cone, and for any $x\in B(x_0,\delta)\cap\overline\Omega$, and any vector $v\in\mathscr I(x_0)$,
    \begin{multline*}
        x+[0,\varepsilon]B(v,\varepsilon)\coloneqq\\
        \{x+\nu u:u\in B(v,\varepsilon),\;0\leq\nu\leq\varepsilon\} \subset\overline\Omega.
    \end{multline*}

The following are the hypotheses required by Theorem \ref{thm:frankowskarampazzo}:
\begin{enumerate}[label=H\arabic*.,ref=H\arabic*]
    \item\label{IPC:first} (Regularity of $F$) \label{IPC:lipschitzity}There exists a constant $K\geq 0$ such that, for every $0<d\leq\T$ and every Lipschitz map $x\colon[0,+\infty)\to\R^n$,
    \begin{equation*} 
    \hspace{-2mm}\int_{0}^{\T-d}\hspace{-4mm}\operatorname{dist}_H(F(s,x(s)),F(s+d,x(s)))\,ds
    \leq Kd.
    \end{equation*}
    Here, $\operatorname{dist}_H$ denotes the Hausdorff metric.
    \item\label{IPC:ipc} (Inward pointing condition) There exist $\eta>0$ and an uniformly hypertangent conical field $\mathscr I$ such that
    \[F(t,x)\cap \mathscr I(x)\cap \{v\in\R^n:|v|\geq \eta\}\neq \emptyset\]
    for all $t\in[0,\T]$ and all $x\in \partial \Omega$.
    \label{IPC:last}
\end{enumerate}
\begin{remark}[Simpler version of \ref{IPC:lipschitzity}]
    If $F(t,x)$ has compact images and is Lipschitz with respect to the Hausdorff metric in $t$, then hypothesis \ref{IPC:lipschitzity} is satisfied; indeed, if the Lipschitz constant were $C>0$, we would have
    \begin{multline*}
        \int_0^{T-d}\dist_H(F(s,x(s)),F(s+d,x(s)))\,ds\\
        \leq\int_0^{T-d} C|s-(s+d)|\,ds=Cd(T-d),
    \end{multline*}
    so taking $K=CT$ we have the condition in \ref{IPC:lipschitzity}. Also, \ref{IPC:lipschitzity} is automatically satisfied when $F$ is autonomous, i.e., when $F$ is independent of time, $F(t,x)=F(x)$.
\end{remark}
\begin{remark}[Simpler versions of \ref{IPC:ipc}]\label{rmk:IPCsimple}
    If $F(t,x)=F(x)$ does not depend on time, $\Omega$ is bounded, and for every $x_0\in\partial \Omega$ one has
    \[F(x_0)\cap \operatorname{interior}(C_{\overline\Omega}(x_0))\neq \emptyset,\] where $C_{\overline\Omega}(x_0)$ denotes Clarke's tangent cone
    at $x_0$, then hypothesis \ref{IPC:ipc} is verified; see \cite[Lem.~4.1]{frankowska2000filippov}.
\end{remark}

\begin{theorem}
 \label{cor:IPC}
 In the context of problems \eqref{min:orig}--\eqref{min:occupation}, let $\domain$ and $\destination$ be open sets, and assume that $\g$ is a continuous function.
  Define $F$ as in \eqref{eq:defUsefulF}. 
 Assume that \ref{IPC:first}--\ref{IPC:last} are verified. %
 Then we have:
 \[\mincurves(\overline\domain,\destination)=\minarcs(\overline\domain,\destination)=\minoccupation(\overline\domain,\destination).\]
 If $\overline \destination=\overline\domain$, then also 
 \begin{equation}\label{eq:mysterious}
 \mincurves(\overline\domain,\overline \destination)=\minarcs(\overline\domain,\overline\destination)=\minoccupation(\overline\domain,\overline\destination).
 \end{equation}
\end{theorem}
\begin{remark} 
    Equality \eqref{eq:mysterious} allows us to apply the  algorithm of \cite{OCP08} in order to obtain $\mincurves(\overline\domain,\overline \destination)$ via the computation of $\minoccupation(\overline\domain,\overline \destination)$ in the case where the system has no endpoint constraint (i.e. $\overline \destination= \overline \domain$). The case where $\destination$ is a closed set strictly contained in $\domain$ may possess a relaxation gap, linked to normality properties of trajectories for the Pontryagin Maximum Principle (see e.g. \cite{palladino2014minimizers}). %

\end{remark}
\begin{remark}
    For an example in which condition \eqref{IPC:ipc} does not hold, giving rise to a relaxation gap, we refer the reader to \cite{OurCounterexamplePaper}.
\end{remark}

As we shall see, Theorem \ref{cor:IPC} follows from the following result.

\begin{theorem}[{\cite[Th.~4.2]{frankowska2000filippov}}]
\label{thm:frankowskarampazzo}
    Assume \ref{IPC:first}--\ref{IPC:last}. Let $\domain\subset\R^n$ be an open set. There is a constant $C>0$ such that the following is true.
    Consider %
    a number $\nu>0$ and an absolutely continuous curve $x\colon[0,\T]\to\overline\Omega$ with $x'(t)\in\operatorname{conv} F(t,x(t))$ for almost every $t$.
    Then there exists a trajectory $y\colon[0,\T]\to\overline\Omega$ such that $y'(t)\in F(t,y(t))$ for almost every $t$, $x(0)=y(0)$, and
    \[\|y(t)-x(t)\|\leq C\nu(1+\T)%
    \qquad \text{for all}\qquad t\in[0,\T].\]
\end{theorem}

\begin{proof}[Proof of Theorem \ref{cor:IPC}]
    The proof of the first set of equalities follows the same lines as the proof of the first part of Theorem \ref{cor:int}, using Theorem \ref{thm:frankowskarampazzo} instead of Theorem \ref{thm:FW} and taking $\nu$ to be small enough that $C\nu(1+\T)\leq \varepsilon$. Observe that, just like in the proof of Theorem \ref{cor:int}, the endpoint $\eta(T)$ is contained in $\destination$ for $\varepsilon$ small enough, by virtue of $X$ being open and $\|\gamma(T)-\eta(T)\|\leq \varepsilon$. The other set of equalities \eqref{eq:mysterious} follows similarly; in this case, upon applying Theorem \ref{thm:frankowskarampazzo} on $\tilde P$ and projecting $y$, we get the curve $\eta$ that ends in $\eta(\T)\in \Omega\subset\overline\Omega=\overline X$, so we again know that $\eta$ is a valid contender.
\end{proof}

\section{Bounds for the gaps}
\label{sec:gapbound}

In this section we briefly discuss practical methods to bound the gap $\mincurves(\overline\domain,\overline \destination)-M_{\mathrm o}(\overline\domain,\overline \destination)$ using the algorithms described in \cite{OCP08}.

Assume that $X$ and $\Omega$ are open, and let $(\domain_\epsilon)_{\epsilon>0}$ and $(\destination_\epsilon)_{\epsilon>0}$ be  families of nonempty compact sets contained in the open sets $\domain_\epsilon\subset \domain$ and $\destination_\epsilon\subset \destination$, and converging to $\domain_\epsilon\to\overline\domain$ and $\destination_\epsilon\to\overline\destination$ in the Hausdorff metric, respectively. For example, one can take
\[\domain_\epsilon=\{x\in \domain:\dist(x,\partial\domain)\geq \epsilon\}\] and\[\destination_\epsilon=\{x\in \destination:\dist(x,\partial \destination)\geq\epsilon\},\]
or, if $\domain=\{x\in\R^n:h(x)> 0\}$ for some continuous function $h$, one can take %
\[\domain_\epsilon=\{x\in\R^n:h(x)\geq \epsilon\},\]
and similarly for $\destination_\epsilon$.
Then $\domain_\epsilon$ and $X_\epsilon$ are interior approximations of $\domain$ and $X$.
\begin{proposition}
Assume that $\destination$ and $\domain$ are open sets.
We have the following bounds for the gaps:
\begin{equation}
\label{gap1}\mincurves(\overline\domain,\overline \destination)-M_{\mathrm o}(\overline\domain,\overline \destination)
\leq M_{\mathrm o}(\domain_{\epsilon},\destination_\epsilon)-M_{\mathrm o}(\overline\domain,\overline \destination).
\end{equation}
\end{proposition}
Observe that in \eqref{gap1}, lower bounds converging to $M_{\mathrm o}(\domain_{\epsilon},\destination_\epsilon)$ are accessible to algorithms like those of \cite{korda2018moments,OCP08} when $h$ and the other data of the problem are polynomial, while upper bounds on $M_{\mathrm o}(\overline\domain,\overline \destination)$ can be computed using  direct or indirect optimal control numerical methods (see, e.g. \cite{caillau2023algorithmic}). 
\begin{proof}
    We have, by the first part of Corollary \ref{cor:int},
\begin{equation*}
    \minarcs(\domain_{\epsilon},\destination_\epsilon) \geq \minarcs(\domain,\destination)=\mincurves(\domain,\destination
    )
    \geq \mincurves(\overline\domain,\destination)
\geq \minarcs(\overline\domain,\overline \destination),
\end{equation*}
because every contender for $\minarcs(\domain,\destination)$ is a contender for $\minarcs(\domain_\epsilon,\destination_\epsilon)$ for all $\epsilon$ small enough.
So we have a bound of the gap (which is the left-hand side): 
\begin{align*} 
&\mincurves(\overline\domain,\destination)-\minoccupation(\overline\domain,\overline \destination)\\
&=
\mincurves(\overline\domain,\destination)-\minarcs(\overline\domain,\overline \destination)\\
&\leq
 \minarcs(\domain_{\epsilon},\destination_\epsilon)-\minarcs(\overline\domain,\overline \destination)\\
&=\minoccupation(\domain_{\epsilon},\destination_\epsilon)-\minoccupation(\overline\domain,\overline \destination).
\end{align*}
Inequality \eqref{gap1} follows from $\mincurves(\overline\domain,\overline\destination)\leq \mincurves(\overline\domain,\destination)$, which is true since all contenders for the latter are also contenders for the former.
\end{proof}

\section{Conclusion and open questions}
\label{sec:conclusions}
We have proposed results ensuring the absence of a relaxation gap exploiting properties from topology, differential geometry, and convex analysis, and discussed a way to bound the gaps. Apart from the general question of characterizing relaxation gaps, other important questions remain open, such as that of whether the gap phenomenon enjoys any genericity or structural stability properties. We hope to be able to address some of these questions in the future.
\appendix
\label{appendix}
\begin{proof}[Proof of Lemma \ref{lem:easyineqs}]
    That $\minoccupation=\minarcs$ follows from \cite[Th.\,19]{patrick}, which shows that every relaxed occupation measure $\mu$ can be decomposed as a convex combination of Young measures together with their corresponding integral curves:~given a contender  $(\mu,\mu_\partial)$ in \eqref{min:occupation}, there is a collection $(\nu^i_t)_{i\in \mathscr I}$ a  probability measure $\eta$ in a (possibly uncountable) set $\mathscr I$ such that, if $\gamma_i$ is the integral curve of $(\nu^i_t)_t$ with $\gamma_i(0)=\initial$, then 
    for every measurable function $\phi(t,x,u)$,
    \begin{equation*}
    \int_{[0,\T]\times\domain\times\controls} \phi(t,x,u)\,d\mu(t,x,u)=
    \int_{\mathscr I}\int_0^{\T}\hspace{-1mm}\int_{\controls} \phi(t,\gamma_i(t),u)\,d\nu_t^i(u)\,dt\,d\eta(i),
    \end{equation*} %
    and for every measurable function $\psi(x)$, 
    \[\int_{\Omega}\psi(x)\,d\mu_\partial(x)=\int_{\mathscr I}\gamma_i(\T)\,d\eta(i).\]
    Since the cost is linear in $(\mu,\mu_\partial)$, the infimum restricted to the set of Young measures must coincide with the infimum on the whole set of relaxed occupation measures.
    
    We turn to show $\minarcs\leq \mincurves$.
    Observe that all the contenders in \eqref{min:orig} are represented by some contender in \eqref{min:relaxed}. Indeed, given a contender $u$ for \eqref{min:orig}, we can take $\nu_t$ to be the atomic measure concentrated at $u(t)$, that is, $\nu_t=\delta_{u(t)}$,
    and then the integral curve of $\nu_t$ is precisely the curve $\gamma$ from the contender pair, because $\int_U f(t,\gamma(t),u)\,d\nu_t(u)=\int_U f(t,\gamma(t),u)\,d\delta_{u(t)}=f(t,\gamma(t),u(t))=\gamma'(t)$ for almost every $t\in[0,T]$, and $\gamma$ is absolutely continuous so $\gamma(t)=\gamma(0)+\int_0^t\gamma'(s)\,ds$.
\end{proof}

\printbibliography
\end{document}